\documentclass[12pt,reqno]{amsart}

\advance \topmargin by -\headheight
\advance \topmargin by -\headsep
\evensidemargin \oddsidemargin
\usepackage[T1]{fontenc}
\usepackage[utf8]{inputenc}
\usepackage{amsmath,amsfonts,amssymb,amsthm}
\usepackage{mathtools}
\usepackage{bm}
\usepackage{enumerate}
\usepackage{booktabs}
\usepackage{microtype}
\usepackage{needspace}
\usepackage{bbm}

\numberwithin{equation}{section}

\newtheorem{theorem}{Theorem}[section]
\newtheorem{proposition}[theorem]{Proposition}
\newtheorem{lemma}[theorem]{Lemma}
\newtheorem{corollary}[theorem]{Corollary}
\newtheorem{fact}[theorem]{Fact}

\theoremstyle{definition}

\newtheorem{example}[theorem]{Example}
\theoremstyle{remark}

\newcommand{\F}{\mathbb F}
\newcommand{\Z}{\mathbb Z}

\newcommand{\e}{\mathbbm{1}}

\title[Projection and fibering]
{Projection and fibering in groups of bounded exponent}
\author{Yifan Jing}
\address{Department of Mathematics, Ohio State University, Columbus, OH USA}
\email{jing.245@osu.edu}

\author{Zuxiang Kong}
\address{Department of Mathematics, Ohio State University, Columbus, OH USA}
\email{kong.570@osu.edu}

\author{Souktik Roy}
\address{Optiver, Chicago, IL USA}
\email{souktik@gmail.com}
\date{}


\begin{document}

\begin{abstract}

We develop a projection and fibering method for sets of small combinatorial doubling in (not necessarily abelian) discrete groups.  As an application, in the abelian case we prove that, if $A$ is finite, the ambient group has exponent $r$, and $|A+A|\leq K|A|$, then
\[
 |\langle A\rangle|\leq r^{(2+o(1))K}|A|.
\]
This answers a question of Ruzsa with optimal leading coefficient, independent to Fox--Pham. The main ingredient is a discrete version of a fiber spillover argument.

For sets in $2$-step nilpotent groups of exponent $r$, we also prove that $|A^3|\leq K|A|$ implies  $|\langle A\rangle|\leq r^{(2+o_K(1))K}|A|$.
The proof combines the abelian theorem with a weighted averaging of central fibers and commutators. 
\end{abstract}

\maketitle

\section{Introduction}

The purpose of this paper is to study how combinatorial doubling of sets performs under suitable projections and fiberings, with applications to Freiman type theorems additive combinatorics. More precisely, let $A\subseteq G$ be a set of small combinatorial doubling, that is $\sigma(A):=|A+A|/|A|\leq K$. Here we assume $A$ is finite and $G$ is abelian for a moment. Let $\pi:G\to H$ be a group homomorphism, and it is well-known that $\sigma(A)\leq K$ does not imply $\sigma(\pi(A))\leq K$. Nevertheless, $\sigma(\pi(A))$ is still under control, and this has been made quantitatively sharp recently by Peng, Tran, and the second author~\cite{KPT24} that $\sigma(\pi(A))\le K^2$. 

Our first proposition is a simple refinement of Theorem 5.1 of \cite{KPT24}. Here $G$ need not to be abelian. 
\begin{proposition}\label{prop: 1}
    Let $G$ be a group, $A,B\subseteq G$ are finite sets such that $|AB|\leq K|B|$. Let $\pi: G\to Q$ be a group homomorphism. Then there is a level set $L\subseteq \pi(B)$ such that $|\pi(A)L|\leq K|L|$. 
\end{proposition}

The proof use the spillover argument~\cite[Proposition 5.3]{AJTZ} which is refined by~\cite{KPT24}. 

As an application of Proposition~\ref{prop: 1}, we answer a question of Ruzsa~\cite{Ruzsa99} on bounding the size of subgroup generated by a set of bounded doubling in a bounded exponent abelian group. Here the exponent of a group $G$ is the least positive integer $r$ such that $g^r=1$ (or $rg=0$ when abelian) for every $g\in G$. Write $\langle A\rangle$ for the group generated by $A$. 

\begin{theorem}\label{thm:main}
Let $G$ be a finite abelian group with exponent $r$, where $r\geq2$, and let $A\subseteq G$ be finite and nonempty.  If
$|A+A|\leq K|A|$, then
\begin{equation}\label{eq: bound}
 |\langle A\rangle|\leq r^{(2+o_K(1))K}|A|.
\end{equation}
\end{theorem}

Ruzsa proved $r^{K^4}$ bound for \eqref{eq: bound} in \cite{Ruzsa99} and he conjectured the actual bound should be $r^{O(K)}$. The bound was improved to $K^2r^{2K^2-2}$ later by Green and Ruzsa~\cite{GR06}. The $\mathbb F_2^n$ case was proven by Green and Tao~\cite{GT09} by introducing a compression argument. The sharp result was obtained by Even-Zohar~\cite{EZ12} and generalized to prime exponents by  Even-Zohar and Lovett \cite{EZL14}, both use the compression argument. Ruzsa's conjecture received many attention in the literature, see for example \cite{DJHP, Diao, HP, Konyagin, Sanders}. It  worth noting that the compression argument requires the existence of a certain lexicographic order in the ambient group, and such an order cannot exist when the exponent is not prime, so the general case requires new argument.

It worth noting that the proof of Theorem~\ref{thm:main} applies Proposition~\ref{prop: 1} to reduce the problem to 
the known finite fields results. However, our projection method can be used to provide a new proof (that is, compression-free) of the earlier known results for finite fields, see discussions in the appendix. 

In 2024, a connection of Ruzsa's conjecture to graph colorings was made by Fox and Pham and they announced a graph theory proof of Theorem~\ref{thm:main}. See the discussion after Theorem 1.8 in \cite{CFPY} and discussions in Section 1.1. Their paper is not yet available when the current paper is written. 

We also want to mention that the constant $2$ in the exponent of the right hand side of \eqref{eq: bound} is optimal. Indeed when  $A=\{0,e_1,\ldots,e_d\}\subseteq(\Z/r\Z)^d$ is a simplex, we have $|A+A|/|A|=(d+2)/2$ and $|\langle A\rangle|=r^d$. 

We next consider the corresponding spanning problem in a nonabelian group. Small doubling alone does not give a bound depending only on the doubling constant and the exponent, even in nilpotency class two, an example is given in Section~6. Small tripling is a natural replacement.

\begin{theorem}\label{thm: main 2-nil}
Let $G$ be a finite $2$-step nilpotent group and exponent $r\geq2$, and let $A\subseteq G$ be nonempty. If
\[
 |A^3|\leq K|A|,
\]
then
\begin{equation}\label{eq: main bound nil}
 |\langle A\rangle|\leq r^{(2+o_K(1))K+2}|A|.
\end{equation}
\end{theorem}

Theorem~\ref{thm: main 2-nil} preserves the leading coefficient in the abelian estimate of Theorem~\ref{thm:main}. We do not claim that the coefficient $2$ is optimal, and the optimal coefficient is likely $1/2$, but the linear dependence on $K$ is optimal, see examples in Section~6. We also remark that replacing $|A^3|\leq K|A|$ by $K$-approximate groups will trivialize the problem (same in abelian settings) as the sharp exponent bound can be obtained via Ruzsa's calculus.

The case general when $G$ in a nonabelian group of bounded exponent is handled by Breuillard, Green, and Tao~\cite[Theorem 6.15]{BGT}. However their proof is not effective and uses Hrushovski's Lie model theorem~\cite{Hrushovski} in an essential way. It would be interesting to see whether bounded exponent case can be handled effectively without Lie model theorem.

\subsection{Abelian case: relation with the Fox--Pham graph theoretic approach}

 Conlon--Fox--Pham--Yepremyan \cite{CFPY} develops a properly edge coloring method for counting small product sets. Its stated additive dimension consequence is weaker than \eqref{eq: bound}. On the other hand, it was mentioned the sharp bounded exponent theorem is included in a subsequent work of Fox--Pham. Pham announced the result at the Edinburgh conference in 2024 and the first author was an audience of that talk. Pham sketched an argument during the talk based on $O(K)$ colors, expanding graph components, and small subsets which realize large sumsets.

The proof here has a different nature. The initial object is to study the difference between the doubling of the original sets and the doubling of its quotients in a quantitative way. The complete Fox--Pham manuscript is not publicly available when the paper is written down.

\subsection{Relation to earlier projects and the Statement of AI}
The project was initiate by the first and the third author back in 2018, and an earlier version of the manuscript contains an error which was later found by Zilin Jiang. The new input is replacing compression by a spillover argument, and the version of discrete groups was recently obtained by Peng, Tran, and the second author~\cite{KPT24}. The ideas are owned by the authors and the paper is written by the authors. ChatGPT is used for language polishing and mainly for error checking after the paper is written up. The figures are drawn by the authors using Ipe. 

\subsection*{Organization}

In Section 2 we record the preliminary inputs. Proposition~\ref{prop: 1} is proved in Section 3 and Theorem~\ref{thm:main} is proved in Section 4. In Section 5, we refine Proposition~\ref{prop: 1} and its corollaries for 2-step nilpotent groups. Theorem~\ref{thm: main 2-nil} is proven in Section 6 and some examples are discussed. In the appendix, we present a self-contained proof of Theorem~\ref{thm:main} for $\F_2^n$ using our projection method. 

\subsection*{Notation and convention}

For finite sets $A,B$ in an abelian group, we write
$A+B=\{a+b:a\in A,\ b\in B\},$
In a multiplicative group, we write $AB=\{ab:a\in A,b\in B\}$ and $A^{-1}=\{a^{-1}:a\in A\}$.  The subgroup generated by a set $A$ is $\langle A\rangle$. We use $\sigma(A)$ for the combinatorial doubling of $A$, that is $|A+A|/|A|$. All logarithms are natural logarithms. In both of the main theorems, the notation $O(K)$ hides an absolute constant and in particular it is uniform in the prime $p$, the exponent $r$, the group, and the set. In the paper we will not distinguish $\mathrm{Aff}(A)$ and $\langle A-A\rangle$ as they only differ by a factor at most $r$ and will be absorbed in the $o_K(1)$ term, we will simply use $\langle A\rangle$ for both. 

\subsection*{Acknowledgment} The authors thanks Zilin Jiang for pointing out an error in an earlier version of the manuscript, and Akshat Mudgal, Chieu-Minh Tran for  discussions on Ruzsa's problem. YJ is supported by NSF grant DMS-2503063.

\section{Preliminaries}\label{sec:preliminaries}

\subsection{Generator rank of abelian groups}

We use the following fact for finitely generated abelian groups. Here for an abelian group $G$ we use $d(G)$ for the minimum number of generators. 

\begin{fact}\label{fact: structure ab}
Let $G$ be a torsion-free abelian group. For a prime $p$, write $pG=\{px:x\in G\}$. The quotient $G/pG$ is naturally a vector space over $\F_p$, with scalar multiplication given by
\[
 a(x+pG)=ax+pG
\]
for all $a\in \F_p$ and $x\in G$. Then
\begin{equation}
 d(G)=\max_{p\mid |G|}\dim_{\F_p}(G/pG).
 \label{eq: generator rank}
\end{equation}
\end{fact}

\begin{proof}
The trivial case is immediate. Otherwise, the structure theorem gives
\[
 G\cong\bigoplus_{i=1}^{t}\Z/n_i\Z,
 \qquad 1<n_1\mid n_2\mid\cdots\mid n_t.
\]
Thus $d(G)\leq t$. For every prime $p$, the images of any generating set of $G$ span $G/pG$ over $\F_p$, so $\dim_{\F_p}(G/pG)\leq d(G)$. Choosing a prime $p_0\mid n_1$, we then have $p_0\mid n_i$ for every $i$, and hence
\[
 G/p_0G\cong\F_{p_0}^t.
\]
Since $p_0\mid |G|$, this gives
\[
 t\leq\max_{p\mid |G|}\dim_{\F_p}(G/pG)\leq d(G)\leq t,
\]
which is \eqref{eq: generator rank}.
\end{proof}

We will use the following elementary counting fact repeatedly.

\begin{lemma}\label{lem: directions}
Let $G$ be a finite abelian group, let $H\leq G$, and $p$ be a prime. Suppose $v_1,\ldots,v_h\in H$ have linearly independent images in $H/pH$.  Let $A$ be a nonempty finite subset of $G$.
If for a set of indices $I\subseteq\{1,\ldots,h\}$ and $i\in I$,
$
 v_i\in\langle A-A\rangle,
$
then $|A|\geq1+|I|$.
\end{lemma}

\begin{proof}
Fix $a\in A$, then the subgroup $W=\langle A-A\rangle$ is generated by the $|A|-1$ differences $b-a$, where $b\in A\setminus\{a\}$. The subgroup $U=\langle v_i:i\in I\rangle\leq W$ requires at least $|I|$ generators in its $p$ primary component, its image in $H/pH$ already has dimension $|I|$. 
\end{proof}

\subsection{Additive combinatorics}

We will use the resolution of the finite field Ruzsa's conjecture as a base case. 

\begin{fact}[Finite field expansion bound]\label{fact:finite-field}
For every prime $p$, every finite dimensional $\F_p$-vector space $V$, and every nonempty $A\subseteq V$ satisfying
$
 |A+A|\leq K|A|,
$
one has
\begin{equation}
 |\langle A\rangle|\leq \frac{p^{2K-2}}{2K-1}|A|.
 \label{eq:finite-field}
\end{equation}
\end{fact}

As discussed in the introduction, for $p=2$, this follows from Even-Zohar's sharp binary theorem \cite{EZ12}. For odd $p$, it follows from Even-Zohar--Lovett \cite{EZL14}.

We use the following standard quantitative consequence of the general Green--Ruzsa theorem for a coarse structure control. Sanders' quantitative structure theorem \cite[Theorem~1.4]{Sanders13} gives a polylogarithmic function of $K$, which is in particular $o(K)$. The more recent improvement of Raghavan may be used in its place.

\begin{fact}[Sanders' Green--Ruzsa]\label{fact: SGR}
There is an absolute constant $C>0$ with the following property. Write $F(K)=C(\log K)^4$, and $A$ is a nonempty finite subset of an abelian group and $|A+A|\leq K|A|$, then there exist a finite set $T$, a subgroup $H$, an integer $d\leq F(K)$, a homomorphism $\phi:\Z^d\to G$ and a centrally symmetric convex body $Q\subseteq\mathbb R^d$ containing $0$
such that, with 
\[
 P=\phi(Q\cap\Z^d),
\]
and $C=H+P$ the coset progression, one has
\begin{equation}
 A\subseteq T+C,\qquad |T|\leq\exp(F(K)),\qquad |C|\leq\exp(F(K))|A|.
 \label{eq: coarse GR}
\end{equation}
\end{fact}

We finish the abelian input by recording the following simple lemma, shows how large can the subgroup generated by a convex progression be when the exponent is bounded. 

\begin{lemma}\label{lem: progression span}
Let $G$ have exponent dividing $r$, let
$C=H+\phi(Q\cap\Z^d)$ be as in Fact~\ref{fact: SGR}, and put
$L=\langle C\rangle$.  Then
\begin{equation}
 |L|\leq |C|r^d.
 \label{eq: progression span}
\end{equation}
\end{lemma}

\begin{proof}
Since $0\in Q$, we have $H\subseteq C$ and hence $|H|\leq|C|$.
Moreover,
\[
 L=H+\langle\phi(Q\cap\Z^d)\rangle
   \leq H+\phi(\Z^d).
\]
The image $\phi(\Z^d)$ is generated by at most $d$ elements, each of
order dividing $r$, and therefore has size at most $r^d$.  Thus
$|L|\leq|H|r^d\leq|C|r^d$.
\end{proof}

\section{A level set refinement of Peng--Kong--Tran}

Let us prove Proposition~\ref{prop: 1}. 
Let $G$ be a (not necessarily abelian) group and we use multiplication as group operation. 
Let $\pi:G\to Q$ a group homomorphism, and $A,B\subseteq G$ are finite, and $|AB|\leq K|B|$. Fix $x\in\pi(A)\subseteq Q$, and define fiber 
\[
A_x = A\cap \pi^{-1}(x). 
\]
Write $a_x = |A_x|$ be the fiber length. Consider the level set
\[
L_j(A) = \{ x\in \pi(A) : a_x \geq j\}.
\]
We similarly define them for $B$. For each $z\in Q$, let
\[
C_z = AB \cap \pi^{-1}(z).
\]
Note that if $z\in \pi(A)L_j(B)$ then $|C_z|\geq j$. Indeed, choose $xy=z$ with $x\in\pi(A)$ and $y\in L_j(B)$, and observe that a translate of $B_y$ lies in $C_z$. Therefore,
\[
\sum_{j\geq1}|\pi(A)L_j(B)|\leq |AB|. 
\]
As $|B| = \sum_{j\geq1}|L_j(B)|$, there is a level set $L$ of $B$ such that
\begin{equation}\label{eq: level set PKT}
\frac{|\pi(A)L|}{|L|}\leq K. 
\end{equation}

Using this, we derive the following result for abelian groups. Recall that $\sigma(L)=|L+L|/|L|$. 
\begin{lemma}\label{lem: bound dim}
    Let $G$ be an abelian group with exponent $r$, and $\phi:G\to V$ a group homomorphism maps $G$ to a $\F_p$-vector space. Suppose $V=\langle \phi(A) - \phi(A)\rangle$ and $|A+A|\leq K|A|$. Then 
    \[
    \dim(V)\leq \frac{\log |\phi(A)|}{\log p} + 2K + O(\log(K)). 
    \]
\end{lemma}
\begin{proof}
    By \eqref{eq: level set PKT}, there is a level set $L$ of $A$ such that $|\phi(A)+L|\leq K|L|$. Let $W=\langle L-L\rangle \leq V$. Now we consider cosets of $W$, which form a partition of $V$. Write $q$ the number of cosets of $W$ that meets $\phi(A)$. 

    As $L\subseteq \pi(A)$, clearly $|\phi(A)+L|\geq |L+L|$. As $L+L\subseteq W$, every other nontrivial cosets of $W$ that meets $\phi(A)$ contribute at least a translate of $L$ in $\phi(A)+L$. Therefore
    \[
    |\phi(A)+L| \geq |L+L| + (q-1)|L|. 
    \]
    From this we get 
    \begin{equation}\label{eq: q and L}
    q \leq 1+ K - \sigma(L)
    \end{equation}
    On the other hand, as $V$ is generated by $\pi(A)-\pi(A)$, $\dim(V) \leq \dim(W) +(q-1)$. By the sharp result for finite field (Fact~\ref{fact:finite-field}), we have
    \[
   |W| = p^{\dim(W)} \leq \frac{p^{2\sigma(L)-2}}{2\sigma(L)-1} |L|. 
    \]
   Since $|L|\leq |\phi(A)|$ and $\sigma(L)+q-1\leq K$ by \eqref{eq: q and L}, we get the desired conclusion. 
\end{proof}

\section{Proof of Ruzsa's conjecture}

\begin{proof}
As $|A+A|\leq K|A|$ and $A\subseteq G$ an abelian group, by Sander's quanitative Green--Ruzsa theorem (Fact~\ref{fact: SGR}), there is a constant $C>0$ such that if we let
\[
F(K) = C(\log K)^{4}=o(K), 
\]
We have $A\subseteq T + P$ where $P$ is a convex progression and $\dim(P)\leq F(K)$, $|T|\leq\exp(F(K))$, and 
\[
|P|\leq \exp(F(K))|A|. 
\]
By losing a factor at most $r$, we may assume that $0\in A$. Write $G=\langle A\rangle$ and $H =\langle P \rangle$, and $Q=G/H$ with $\pi: G\to Q$ the natural projection. $|\pi (A)|\leq T$ by Fact~\ref{fact: SGR}. Hence by Lemma~\ref{lem: progression span}, we have
\begin{equation}\label{eq: bound for H}
|H|\leq |P| r^{\dim(P)}. 
\end{equation}
Now for prime $p$ divides $r$, consider $\psi: Q \to Q/pQ$ and $\phi = \psi\circ \pi$, that is
\[
\phi: G \xrightarrow{\pi}  Q \xrightarrow{\psi}  Q/pQ.
\]
By Lemma~\ref{lem: bound dim}, 
\[
\dim_{\F_p}(Q/pQ) \leq \frac{F(K)}{\log p} + 2K +O(\log K) = (2+o(1))K. 
\]
Then by the structure theorem (Fact~\ref{fact: structure ab}), 
\[
d(Q) = \max_{p\mid r} \dim_{\mathbb F_p}(Q/pQ)\leq (2+o(1))K,
\]
and this implies $|Q|\leq r^{d(Q)}= r^{(2+o(1))K}$. Finally, by \eqref{eq: bound for H}, 
\[
|G| \leq |H||Q| \leq r^{(2+o(1))K} |A| 
\]
as desired. 
\end{proof}

Let us remark that in the proof we use Sanders' quantitative Green--Ruzsa theorem (Fact~\ref{fact: SGR}) as a coarse structure input. We believe this input can be removed by iteratively applying Proposition~\ref{prop: 1}, and it is then possible to obtain a sharp control on the $o_K(1)$ term. See the appendix, where we proved a special case without Sanders' coarse structure theorem being used. Removing Sanders' input from the proof of Theorem~\ref{thm:main} would be an interesting direction, but we do not pursue it in this paper.


\section{A refinement of the projection argument in 2-step nilpotent groups}

Throughout this section, let $N=\langle A\rangle$ be a finite
group of nilpotency class at most two and with exponent $r\geq2$, where $A$ is nonempty and $|A^3|\leq K|A|.$ In fact, assuming elements in $A^2$ having exponent $r$ suffices. 

Let $H=N/[N,N]$ be the abelianization and $\pi: N\to H$ the natural projection. We use additive notation in $H$ and multiplicative notation in $N$. 

As usual, for each $x\in\pi(A)$ we choose $a_x\in A\cap\pi^{-1}(x)$ and write 
\[
A_x = A\cap \pi^{-1}(x) = a_xU_x
\]
for some $U_x\subseteq[N,N]$ containing the identity. Write $f(x)=|U_x|$ the fiber length. We use 
\[
L_t = \{x\in\pi(A): f(x)\geq t\}
\]
for the level set. 

Now we consider the largest fiber, that is, we pick $x_0\in H$ maximizing $f$, and let $A_0=A_{x_0}$ and $U_0=U_{x_0}$. Let  $W = \langle \bigcup_x U_x\rangle$ and $W_0 = \langle U_{0}\rangle$ which are subgroups of $[N,N]$. Hence $W$ captures the fiber information of $A$ in $[N,N]$. The following proposition refines Lemma~\ref{lem: bound dim}. As $G$ is nonabelian we estimate the product set from projection and fibers simultaneously. 

\begin{proposition}\label{nc:prop:joint}
With notation fixed above, for a prime $p$, write
\[
 \ell_p=\dim_{\F_p}\bigl((W/W_0)/p(W/W_0)\bigr),
 \qquad
 m_p=\dim_{\F_p}\bigl(([N,N]/W)/p([N,N]/W)\bigr).
\]
Then
\begin{equation}\label{eq: joint full nil}
 |A^3|\geq \sum_{t=1}^{|U_{0}|}|\pi(A)+\pi(A)+L_t|
 +m_p|A|+\ell_p|U_{0}||\pi(A)|+(|U_0^2|-|U_0|)|\pi(A)|.
\end{equation}
\end{proposition}

\begin{proof}
Let us first pick the basis in the two $\F_p$-vector spaces  $(W/W_0)/p(W/W_0)$ as well as  $([N,N]/W)/p([N,N]/W)$.  For $W/W_0$, the images of $U_x$ generate it, then we may choose $u_i\in U_{x_i}$ for $1\leq i\leq \ell_p$ such that the image of $\{u_i\}$ generates $(W/W_0)/p(W/W_0)$. Note that here $U_{x_i}$ need not be distinct. 

As $N$ is $2$-step nilpotent, $[N,N]$ is central. Then for each $u,v\in [N,N]$ one has
\[
[a_xu,a_yv] = [a_x, a_y].
\]
Now we choose $(x_i,y_i)\in\pi(A)\times\pi(A)$ with $1\leq i\leq m_p$ such that
\[
c_i = [a_{x_i},a_{y_i}]W
\]
form a basis in the vector space $([N,N]/W)/p([N,N]/W)$.

Let us now consider the coset of $W$ in  $N$. For each $W$ coset $S$ in $N$ we have the following trivial identity
\[
|A^3\cap S| = \min\{|A^3\cap S|, |U_0|\} + (|A^3\cap S|-|U_0|)_+ =: \mathrm{I} + \mathrm{II},
\]
and we will estimate two terms separately. 

\begin{figure}[h]
    \centering
    \includegraphics[width=0.7\linewidth]{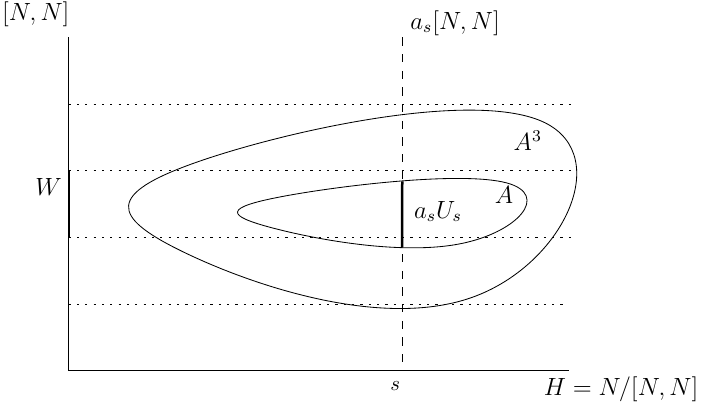}
    \caption{Grid type decomposition of $N$ using cosets of $W$}
    \label{fig:nilp}
\end{figure}

\noindent{\emph{Term I estimates.}} For $s\in H$ and $1\leq t\leq |U_0|$, consider the following notion of level set which counts the number of $W$ cosets with at least $t$ elements in $A^3$ fiberwise: 
\[
n_s(t) = \# \{S\subseteq \pi^{-1}(s): S \text{ is a }W\text{-coset, }|A^3\cap S|\geq t\}.
\]
Recall the choice of $c_i = [a_{x_i},a_{y_i}]W$ we defined above. 
Let $I_s(t)\subseteq [m_p]$ so that for each $i\in I_s(t)$, if we write $z=s-x_i-y_i$, then $z\in L_t$, that is, $\pi^{-1}(z)$ is a $[N,N]$-fiber with at least $t$ element of $A$. Observe that $a_{x_i}a_{y_i}A_z$ and $a_{y_i}a_{x_i}A_z$ both lie in $A^3\cap\pi^{-1}(s)$, each has at least $t$ elements, and each is contained in a $W$ coset, and these cosets differ by $c_i$. 

The cosets counted in $n_s(t)$ can be viewd as subsets of $[N,N]/W$. If $I_s(t)$ is not empty, the above argument gives a $c_i$ from the difference. By Lemma~\ref{lem: directions}, we have 
\[
n_s(t)\geq 1+|I_s(t)|. 
\]
Hence if $s\in \pi(A)+\pi(A)+L_t$, then a set $a_xa_yA_z$ with $x,y\in \pi(A)$ and $z\in L_t$, and with $s = x+y+z$, gives at least one coset count by $n_s(t)$. As $I_s(t)$ is empty whenever $s\notin \pi(A)+\pi(A)+L_t$ we have
\[
n_s(t)\geq \e_{\pi(A)+\pi(A)+L_t}(s) + |I_s(t)|. 
\]
As for each fixed $i\in[m_p]$ there are exactly $L_t$ many $s$ contribute to $|I_s(t)|$, therefore
\[
\sum_{s\in H} n_s(t) \geq |\pi(A)+\pi(A)+L_t| + m_p|L_t|. 
\]
Use the fact that  $\sum_{t=1}^{|U_0|} |L_t|=|A|$, we have
\[
\sum_S \min\{|A^3\cap S|,|U_0|\} = \sum_{t=1}^{|U_0|} \sum_{s\in H} n_s(t) \geq \sum_{t=1}^{|U_0|}|\pi(A)+\pi(A)+L_t| + m_p|A|. 
\]

\noindent{\emph{Term II estimates.}}
Recall that for $1\leq j\leq \ell_p$ we have chosen $\{u_j\}$ that their images generates $(W/W_0)/p(W/W_0)$, where $u_i\in U_{x_i}$.
For each $1\leq j\leq \ell_p$ and $y\in \pi(A)$, consider $a_ya_{x_j}A_0$ and $a_ya_{x_j}u_jA_0$. Both lies in $A^3$ and has $|U_0|$ elements. Since $u_j$ is central and lies in $W$, these two sets are contained in a same $W$ coset, and their $W_0$ coset differ by $u_jW_0$.  

Now we fix a $W$ coset $C$ in $N$ and let $e_C$ be the number of pairs of sets $(a_ya_{x_j}A_0,a_ya_{x_j}u_jA_0)$ that both lie in $C$. Notice that each indices $j$ appears at most once in $C$. Indeed, $y+x_j+x_0$ locates the fiber that contains $C$. Hence inside each $C$, the directions $u_j$ associated with each pair are independent modulo $p$ in $W/W_0$. If $e_C>0$, applying Lemma~\ref{lem: directions} inside each $C$ gives that it contains at least $e_C+1$ distinct $W_0$ cosets, each containing at least $|U_0|$ elements of $A^3$. Therefore
\[
(|A^3\cap C| - |U_0|)_+ \geq |U_0| e_C. 
\]
The above inequality is immediate when $e_C=0$. 

On the other hand, for every $y\in \pi(A)$, as $[N,N]$ in central, $A_0A_0a_y = a_{x_0}^2a_yU_0^2$. It has $|U_0^2|\geq |U_0|$ elements, and lies in a $W_0$-coset, and lies on $\pi^{-1}(2x_0+y)$. Let $C_y$ be the $W$-coset that contains it. If this $W_0$-coset is counted above, its contribution increases from $|U_0|$ to $|U_0^2|$. If this is a new coset then the contribution is additional. In either case we have
\[
|C_y\cap A^3| \geq |U_0|e_{C_y} + |U_0^2|. 
\]
Thus
\[
(|A^3\cap C_y| - |U_0|)_+ \geq |U_0| e_{C_y}+|U_0^2|-|U_0|.
\]
As the location $2x_0+y$ are distinct when $y$ varies, all cosets $C_y$ are distinct. Also, there are exactly $\ell_p|\pi(A)|$ pairs $(a_ya_{x_j}A_0,a_ya_{x_j}u_jA_0)$, hence
\[
\sum_C (|A^3\cap C|-|U_0|)_+ \geq |U_0|\ell_p|\pi(A)| + (|U_0^2|-|U_0|)|\pi(A)|. 
\]
Combining the two estimates for I and II finishes the proof. 
\end{proof}

We have the following corollary. Recall that $\sigma(\Omega)$ is defined by $|\Omega+\Omega|/|\Omega|$ when $\Omega$ is a subset of abelian group, and $|A^3|/|A|\leq K$, $N$ is a group of exponent $r$. 

\begin{corollary}\label{cor: two estimates}
    With notation defined in Proposition~\ref{nc:prop:joint},  we have
    \[
    |[N,N]/W_0|\leq r^{K - \frac{|U_0||\pi(A)|}{|A|} \big(\sigma(\pi(A)) + \sigma(U_0) -1 \big)}
    \]
\end{corollary}

\begin{proof}
For every $1\leq t\leq |U_0|$, we have $x_0\in L_t$, and hence $|\pi(A)+\pi(A)+L_t|\geq|\pi(A)+\pi(A)|$. Dividing \eqref{eq: joint full nil} by $|A|$ therefore gives
\[
m_p + \frac{|U_0||\pi(A)|}{|A|}\big(\sigma(\pi(A)) + \sigma(U_0) -1 + \ell_p\big) \leq K. 
\]
The exact sequence of abelian groups
\[
 0\to W/W_0\to [N,N]/W_0 \to [N,N]/W\to0
\]
gives in the $p$-primary components that
\[
 d(([N,N]/W_0)_p)\leq\ell_p+m_p
 \leq\frac{|U_0||\pi(A)|}{|A|}\ell_p+m_p
 \leq K-\frac{|U_0||\pi(A)|}{|A|}\big(\sigma(\pi(A)) + \sigma(U_0) -1 \big).
\]
This holds for every prime $p$. The desired conclusion is implied by Fact~\ref{fact: structure ab}. 
\end{proof}

\section{Proof of the $2$-step nilpotent theorem}

In this section, we prove Theorem~\ref{thm: main 2-nil} and provide some examples. We will use the same notation used in Proposition~\ref{nc:prop:joint}, that $H=N/[N,N]$, $A_0=a_0U_0$ the largest $[N,N]$ fiber of $A$. $W_0=\langle U_0\rangle$ and $W=\langle\bigcup U_{x_i}\rangle $ where $a_{x_i}U_{x_i} = \pi^{-1}(x_i)\cap A$. 

\begin{proof}[Proof of Theorem~\ref{thm: main 2-nil}]
As $H=\langle \pi(A)\rangle$ and $W_0 = \langle U_0\rangle$, Theorem~\ref{thm:main} gives 
\[
\frac{|H|}{|\pi(A)|}\leq r^{(2+\varepsilon)\sigma(\pi(A))},\qquad \text{and}\qquad \frac{|W_0|}{|U_0|}\leq r^{(2+\varepsilon)\sigma(U_0)}. 
\]
Observe the following identity
\[
\frac{|N|}{|A|} = \frac{|U_0||\pi(A)|}{|A|} \frac{|H|}{|\pi(A)|}\frac{|W_0|}{|U_0|} |[N,N]/W_0|. 
\]
Hence by Corollary~\ref{cor: two estimates}, write $k= \sigma(\pi(A))+\sigma(U_0)$ we get
\begin{equation}\label{eq: key nil}
\log_r \frac{|N|}{|A|}\leq \log_r \frac{|U_0||\pi(A)|}{|A|} + (2+\varepsilon)k + K - (k-1)\frac{|U_0||\pi(A)|}{|A|}. 
\end{equation}
Write $\rho = |U_0||\pi(A)|/|A|$. Then $1\leq \rho  \leq K$ as $|A^2|\leq |A^3|\leq K|A|$, and Corollary~\ref{cor: two estimates} also implies
\[
2\leq k\leq \frac{K}{\rho } +1. 
\]
We now split the computation in two cases. 

If $\rho  \leq 2+\varepsilon$, the coefficient of $k$ in \eqref{eq: key nil} is nonnegative, and one can apply the upper bound on $k$ to conclude the desired bound. Here we also use that $r\geq 2$ and $\rho\geq 1$ to conclude $\log_r \rho \leq 2(\rho-1)\leq 2K(1-1/\rho)$. 

If $\rho > 2+\varepsilon$, we use $k\geq 2$ instead in \eqref{eq: key nil}, and the fact that $\log_r\rho\leq \rho-1$. In either case we have
\[
\frac{|N|}{|A|}\leq r^{(2+o(1))K+2}
\]
by letting $\varepsilon\to 0$. 
\end{proof}

The following example shows that the linear dependence on $K$ on the exponent of Theorem~\ref{thm: main 2-nil} is necessary. 

\begin{example}
Fix a prime $p\geq5$.  Let $V=\F_p^d$ and
$Z=\bigwedge^2 V$.  On $N=V\times Z$ define
\[
 (v,z)(w,t)=\left(v+w,z+t+\frac12 v\wedge w\right).
\]
Bilinearity gives associativity, and it is easy to check that this is a 2-step nilpotent group of exponent $p$, with central commutators
\[
 [(v,z),(w,t)]=(0,v\wedge w).
\]
The elements $x_i=(e_i,0)$ generate $N$, and
\[
 |N|=p^{d+\binom d2}.
\]
Let $A=\{1,x_1,\dots,x_d\}$.

Note that all positive words of length at most three are distinct. Indeed, the $V$-coordinate determines their multiplicities as $p>3$. For three distinct letters, the $Z$-coordinate records every pairwise order and hence their order in the word.  For two occurrences of one letter and one of another, the three possible positions give distinct wedge coefficients $1,0,-1$. The cases of length at most two and of a single repeated letter are immediate. Therefore
\[
 |A^3|=1+d+d^2+d^3=(d+1)(d^2+1),
\]
which implies $K=|A^3|/|A|=d^2+1$. 
Therefore
\[
 \log_p\frac{|N|}{|A|}
 =\left(\frac12+o(1)\right)K.
\]
Thus an upper bound $r^{o(K)}|A|$ is impossible even for a fixed odd prime exponent.
\end{example}

The next example shows that if we assume small doubling instead of small tripling, that is $|AA|/|A|$ being small, $|\langle A\rangle|/|A|$ can be unbounded even in the $2$-step nilpotent setting. This is a classic subgroup union a generic point example. 

\begin{example}
Fix an odd prime $p$.  Consider the 2-step nilpotent group with exponent $p$ generated by $h_1,\ldots,h_d,x,z_1,\ldots,z_d$ with $[h_i,x]=z_i$, where the $z_i$ are independent central elements and all other commutators between the displayed generators are trivial.
For example, it is realized on
$\F_p^d\times\F_p\times\F_p^d$ by
\[
 (u,t,z)(v,s,w)=(u+v,t+s,z+w+s u).
\]
Let $H=\langle h_1,\ldots,h_d\rangle$ and
$A=H\cup\{x\}$.  Then we have $|H|=p^d$, $|N|=p^{2d+1}$, and $H\cap H^x = 1$. 
The sets $H$, $Hx\cup xH$, and $\{x^2\}$ are disjoint. Hence
\[
 |A^2|=3|H|<3|A|,
 \]
 and
 \[
 \frac{|N|}{|A|}=\frac{p^{2d+1}}{p^d+1}\to\infty.
\]
\end{example}

\appendix
\section{A direct proof of Green--Tao's expansion in $\F_2^n$ theorem}

The proof of Theorem~\ref{thm:main} uses the known Ruzsa's problem for finite fields (in other word, an abelian group with exponent being a prime) as an input. In fact, our projection method is able to provide a different proof of these results, by inductively applying Proposition~\ref{prop: 1}. In the appendix, we prove the version for $\F_2^n$ as the computation is simpler, and we leave the $\F_p^n$ version for the interested readers.

\begin{theorem}
    Let $A\subseteq \F_2^n$ and $|A+A|\leq K|A|$. Then 
    \[
   | \langle A-A\rangle |\leq 2^{2K-2} |A|. 
    \]
\end{theorem}

\begin{proof}
   Write $\langle A-A\rangle=\F_2^d$ and we will show
   \[
   d\leq \log_2|A| +2\sigma(A)-2
   \]
   inductively. 

   For each $u\in A+A$ as usual write $r(u)= |A\cap u-A|$. Pick $u_0$ such that $r(u_0)$ minimize $r$. Hence $r(u_0)\leq |A|^2/|A+A|$ by averaging. 
   Let $U = \langle u_0\rangle$ and consider the natural projection $\pi:\F_2^d\to \F_2^d/U$. As fibers have length $2$, we have $|\pi(A)| =|A| - r(u_0)/2$. 

   We now use the proof of Proposition~\ref{prop: 1}. Pick $a+b=u_0$ with $a,b\in A$, and for every $c\in \pi(A)$, write $[c]\in A$ the representative element, $[c]+a$ and $[c]+b$ both lie in $A+A$ and lives in a same $U$-coset. Hence at least $|\pi(A)|$ many $U$-fibers of $A+A$ are of full length. Therefore
   \[
   |A+A| - |\pi(A)+\pi(A)|\geq |\pi(A)|.
   \]
   Using this we have
   \[
   \sigma(A)-\sigma(\pi(A))\geq \frac{ |A+A| - |\pi(A)+\pi(A)|-\sigma(A)(|A|-|\pi(A)|)}{|\pi(A)|}
   \]
   Note that $|A|-|\pi(A)|=r(u_0)/2\leq |A|^2/2|A+A|$, we have
   \[
   2(\sigma(A)-\sigma(\pi(A)))\geq 2-|A|/|\pi(A)|.
   \]
   As $|A|/|\pi(A)|$ is in $[1,2]$, hence $|A|/|\pi(A)| -1 \leq \log_2(|A|/|\pi(A)|)$. Therefore,
   \[
   2(\sigma(A)-\sigma(\pi(A))) + \log_2\frac{|A|}{|\pi(A)|}\leq 1. 
   \]
   Now, apply induction hypothesis on $\pi(A)$ and $\F_2^d/U$ gives
   \[
   d-1\leq \log_2|\pi(A)| + 2\sigma(\pi(A)) -2. 
   \]
   Plug in the above estimates for $(\sigma(A)-\sigma(\pi(A))$ finishes the proof.
\end{proof}

\bibliographystyle{amsalpha}
\bibliography{reference}

\end{document}